\documentclass[11pt]{amsart}

\usepackage[T1]{fontenc}
\usepackage{lmodern}
\usepackage{amsmath,amssymb,mathtools}
\usepackage{enumitem}
\usepackage{microtype}
\usepackage[colorlinks=true,linkcolor=blue,citecolor=blue,urlcolor=blue]{hyperref}

\allowdisplaybreaks
\numberwithin{equation}{section}
\hypersetup{
  pdftitle={Limit laws for component-pruned sparse random graphs and percolated tori},
  pdfauthor={Mostafa Mirabi and Saharon Shelah}
}

\newtheorem{theorem}{Theorem}[section]
\newtheorem{lemma}[theorem]{Lemma}
\newtheorem{proposition}[theorem]{Proposition}
\newtheorem{corollary}[theorem]{Corollary}

\theoremstyle{definition}
\newtheorem{definition}[theorem]{Definition}
\newtheorem{remark}[theorem]{Remark}

\newcommand{\N}{\mathbb N}
\newcommand{\Z}{\mathbb Z}
\newcommand{\Prob}{\mathbb P}
\newcommand{\Ex}{\mathbb E}
\newcommand{\Var}{\operatorname{Var}}
\newcommand{\Aut}{\operatorname{Aut}}
\newcommand{\Prune}{\operatorname{Prune}}
\newcommand{\qr}{\operatorname{qr}}
\newcommand{\Po}{\operatorname{Poisson}}

\title[Limit laws for pruned sparse graphs and tori]{Limit laws for component-pruned sparse random graphs and percolated tori}

\author{Mostafa Mirabi}
\address{\newline The Taft School, Watertown, CT 06795, USA and \newline  Wesleyan University, Middletown, CT 06459, USA}
\email{mmirabi@wesleyan.edu}
\urladdr{https://sites.google.com/site/mostafamirabi}

\author{Saharon Shelah}
\address{\newline Hebrew University of Jerusalem, Jerusalem, Israel}
\email{saharon.shelah@mail.huji.ac.il}
\urladdr{https://shelah.logic.at}


\thanks{The second author research partially supported by the Israel Science Foundation
(ISF) grant no: 1838/19, and Israel Science Foundation (ISF) grant no: 2320/23; Research partially
supported by the grant “Independent Theories” NSF-BSF, (BSF 3013005232). This is publication
number 1277 of the second author.}

\begin{document}
\raggedbottom

\begin{abstract}
We prove an $\mathrm{MSO}_2$ zero--one law for a very sparse
Erd\H{o}s--R\'enyi graph after pruning by component order.  Let $p_n=c_n/n$,
where $c_n\to0$, and delete every component of order less than $f(n)$, where
$f(n)\to\infty$.  If
\[
        f(n)\bigl(\log f(n)+\log(1/c_n)\bigr)=o(\log n),
\]
then the resulting graph satisfies a zero--one law for $\mathrm{MSO}_2$,
with quantification over sets of vertices and sets of edges.  The proof
combines uniform component counts, an MSO Feferman--Vaught decomposition for
disjoint unions, and semilinearity of the order spectra of MSO-definable
classes of finite trees.  We also show that the term $f(n)\log f(n)$ cannot
simply be omitted: star components can occur at first-order-visible Poisson
thresholds.

We further establish first-order limit laws for bond percolation on the
discrete torus $T_L^d$.  In the two-sided subpolynomial regime, pruning below
a sufficiently slow threshold yields a zero--one law.  For the unpruned model
in either one-sided polynomial regime, the reciprocal exponents
$\alpha=1/k$ are precisely the critical scales.  At such a scale, an extended
limit of $N p_N^k$ or $N q_N^k$ equal to $0$ or $\infty$ gives a zero--one law;
a positive finite limit gives a convergence law but not a zero--one law; and
the absence of an extended limit gives failure of convergence.  Finally,
$\mathrm{MSO}_1$ already detects the parity of the torus side length through
bipartiteness, producing a natural obstruction to monadic convergence in a
near-deterministic regime.
\end{abstract}

\maketitle

\section{Introduction}

The Erd\H{o}s--R\'enyi model is a basic meeting point of probability,
combinatorics, and logic \cite{ER}.  A classical starting point is the
zero--one law for the dense random graph, proved independently by Glebskii,
Kogan, Liogon'kii, and Talanov and by Fagin \cite{Glebskii,Fagin}: for fixed
edge probability, every first-order graph property has limiting probability
$0$ or $1$.  Sparse random graphs are subtler because local configurations
appear at different scales.  The work of Shelah--Spencer and Lynch gives a
detailed model-theoretic analysis in important sparse regimes
\cite{ShelahSpencer,Lynch}.  Locality is central to many such results: once
the quantifier rank is fixed, bounded neighborhoods and their multiplicities
control the relevant first-order information.  We use the finite
bounded-degree form of Hanf locality for this purpose; see
\cite{Hanf,FaginStockmeyerVardi,Gaifman,Libkin}.

This paper studies logical limit laws in two related sparse settings.  The
first is obtained by pruning according to component order.  For a finite
graph $G$ and $m\geq1$, define
\[
   \Prune_m(G)=
   G\upharpoonright\{v\in V(G):|C_G(v)|\geq m\},
\]
where $C_G(v)$ is the connected component containing $v$.  Thus
$\Prune_m(G)$ is obtained by deleting every component of order less than
$m$.  This operation is unrelated to the usual graph-theoretic $k$-core,
which iteratively deletes vertices of small degree.  For a sequence $(G_n)$
and an integer-valued function $f$, we abbreviate
$\Prune_{f(n)}(G_n)$ to $\Prune_f(G_n)$.

We first apply this operation to $G(n,c_n/n)$ with $c_n\to0$.  Before
pruning, the graph is a forest with probability tending to one, and its tree
components can be counted by standard first- and second-moment estimates
\cite{Bollobas,JLR}.  After pruning with $f(n)\to\infty$, no fixed finite
component remains.  At each fixed monadic quantifier rank, every component
type of bounded order disappears, whereas every type having arbitrarily
large tree representatives occurs many times.  An MSO Feferman--Vaught
composition theorem then turns this deterministic component profile into an
$\mathrm{MSO}_2$ zero--one law.

The term $f(n)\log f(n)$ in our hypothesis is essential for this argument.
First-order logic can detect a star component by asking for a vertex all of
whose neighbors are leaves.  If $t=t(n)\to\infty$ and $t=o(n^{1/2})$, the
expected number of star components on $t$ vertices in $G(n,c_n/n)$ is
\[
       (1+o(1))\,n\frac{c_n^{t-1}}{(t-1)!}e^{-c_n t}.
\]
Consequently, the weaker condition
$f(n)\log(1/c_n)=o(\log n)$ still permits a first-order-visible Poisson
threshold at the pruning boundary.  The stronger assumption
\[
       f(n)\bigl(\log f(n)+\log(1/c_n)\bigr)=o(\log n)
\]
excludes this obstruction, and we give an explicit counterexample under the
weaker condition.

Our second setting is bond percolation on the discrete torus
\[
   T_L^d=(\Z/L\Z)^d,
\]
where adjacent vertices differ by $\pm1$ in exactly one coordinate.  Write
$N=L^d$.  In the random subgraph $X_L^d(q_N)$, every torus edge is retained
with probability $q_N$ and omitted with probability $p_N=1-q_N$,
independently.  Because the degree is bounded, Hanf locality applies
directly.  In the two-sided subpolynomial regime, a slowly growing pruning
threshold removes all bounded components while leaving many copies of every
local type that can occur in an infinite subgraph of $\Z^d$.  This yields a
first-order zero--one law.

We also classify the unpruned model in the one-sided polynomial regimes.
Suppose, for example, that $p_N\to0$ and $q_N\to1$.  Each bounded rooted
neighborhood type has a defect number: the least number of omitted edges
needed to produce it from the full lattice neighborhood.  A type of defect
$b$ has expected count of order $N p_N^b$.  Hence the reciprocal exponents
$\alpha=1/k$ are the critical scales when
$p_N=N^{-\alpha+o(1)}$.  Away from those exponents, every local type is either
absent or saturated.  At a reciprocal exponent, finite-range defect clusters
have Poisson limits, while rooted type counts are bounded cluster functions
of those Poisson variables.  This gives convergence at a finite critical
parameter, zero--one laws at the two extended endpoints, and failure of
convergence when the critical parameter itself does not converge.

The monadic behavior on tori is different.  Although component composition
controls $\mathrm{MSO}_2$ on forests, monadic logic can express global
properties beyond first-order locality on tori.  In fact,
$\mathrm{MSO}_1$ already expresses bipartiteness and therefore detects the
parity of long coordinate cycles.  This yields a direct obstruction to
monadic convergence when the side length ranges over both even and odd
integers.

\subsection*{Relation to previous work}

The local-pattern analysis on tori belongs to an established line of work on
logical laws for random lattice fields.  Coupier, Desolneux, and Ycart proved a
first-order zero--one law for independent random images by combining Gaifman
locality with threshold estimates for finite pixel patterns
\cite{CoupierDesolneuxYcartImages}.  Their work on statistical area
thresholding also studies the removal of small connected components as an
image-denoising operation, using Poisson approximation to calibrate the area
cutoff \cite{CoupierDesolneuxYcartArea}.  Coupier, Doukhan, and Ycart extended
the logical local-pattern method to dependent binary fields on lattice tori
under mixing assumptions \cite{CoupierDoukhanYcart}.  Our first-order torus
arguments use this same locality-and-counting mechanism.  The features
specific to our torus theorems are the bond-percolation graph language,
pruning of percolation components, and the explicit organization of the
one-sided critical regimes by the defect quantities $N p_N^k$ and $N q_N^k$.
Thus the unpruned polynomial classification is an edge-percolation refinement
of this framework rather than a new general locality principle.

There is also a substantial literature on monadic limit laws for random
tree-like graphs.  Woods obtained MSO limiting probabilities for random
rooted trees through finite coloring rules \cite{Woods}; McColm proved MSO
zero--one laws for random labeled acyclic graphs \cite{McColm}.  More broadly,
Heinig, M\"uller, Noy, and Taraz established an MSO zero--one law for connected
random graphs and an MSO convergence law for unrestricted random graphs from
addable minor-closed classes \cite{HeinigMullerNoyTaraz}.  On the negative
side, Kaufmann and Shelah showed that the MSO zero--one law can fail very
strongly for uniform random finite relational structures
\cite{KaufmannShelah}.  For sparse Erd\H{o}s--R\'enyi graphs,
Ostrovsky and Zhukovskii studied bounded-quantifier-depth MSO laws in the very
sparse power-law range \cite{OstrovskyZhukovskii}.  The present result has a
different source: a deterministic component-size filter is applied to
$G(n,c_n/n)$ with $c_n\to0$, and the conclusion is a zero--one law for every
$\mathrm{MSO}_2$ sentence, including quantification over edge sets.  Its proof
uses the MSO disjoint-union composition theorem and regular-tree size spectra,
rather than a new local normal form for MSO.

\subsection*{Main results}

For the Erd\H{o}s--R\'enyi part, write $G(n,p)$ for the random graph on
$[n]$ in which edges are present independently with probability $p$.  The
principal theorem of the paper is the following; its proof is completed in
Section~\ref{sec:mso} after the required MSO composition tools are developed.

\begin{theorem}\label{thm:ER-mso}
Let $p_n=c_n/n$, where $c_n>0$ and $c_n\to0$.  Let
$f:\N\to\N$ satisfy $f(n)\to\infty$ and
\[
        f(n)\bigl(\log f(n)+\log(1/c_n)\bigr)=o(\log n).
\]
Put
\[
        H_n=\Prune_{f(n)}(G(n,p_n)).
\]
Then $(H_n)$ has an $\mathrm{MSO}_2$ zero--one law.  In particular, it has an
$\mathrm{MSO}_1$ zero--one law and a first-order zero--one law.
\end{theorem}

The displayed hypothesis already implies
$\log(1/c_n)=o(\log n)$, so no separate subpolynomial assumption on $c_n$ is
needed.

We first isolate the first-order component profile, both as a warm-up and for
later comparison with the bounded-degree torus argument.

\begin{theorem}\label{thm:ER-core}
Let $p_n=c_n/n$, where $c_n>0$ and $c_n\to0$.  Let
$f:\N\to\N$ satisfy $f(n)\to\infty$ and
\[
       f(n)\bigl(\log f(n)+\log(1/c_n)\bigr)=o(\log n).
\]
Put
\[
       H_n=\Prune_{f(n)}(G(n,p_n)).
\]
Then $(H_n)$ has a first-order zero--one law.  More precisely, for every
quantifier rank $s$ and every fixed $K$, with probability tending to one
$H_n$ is a forest, has no component of any bounded $s$-component type, and
has at least $K$ components of every unbounded $s$-component type.
Consequently, for every $s$ there is a deterministic $s$-theory of forests
which $H_n$ has with probability tending to one.  This is the universal
large-component forest profile described in
Proposition~\ref{prop:forest-profile}.
\end{theorem}

For the torus, fix $d\geq1$ and let
\[
   T_L^d=(\Z/L\Z)^d
\]
with an edge between two vertices if they differ by $1$ or $-1$ modulo $L$ in
exactly one coordinate and agree in all other coordinates.  The number of
vertices is $N=L^d$.  In bond percolation on $T_L^d$, each edge is retained
with probability $q_N$ and omitted with probability $p_N=1-q_N$,
independently.  The resulting random subgraph is denoted $X_L^d(q_N)$.

\begin{theorem}\label{thm:torus-core}
Fix $d\geq1$.  Let $N=L^d$ and let $X_L=X_L^d(q_N)$.  Put
\[
        \eta_N=\min\{p_N,q_N\}.
\]
Let $f:\N\to\N$ be integer-valued.  Assume $\eta_N>0$ eventually,
$f(N)\to\infty$, and
\[
        f(N)\log(1/\eta_N)=o(\log N).
\]
Then
\[
        Y_L=\Prune_{f(N)}(X_L)
\]
has a first-order zero--one law.  Its almost-sure first-order theory is
determined by the following axiom schemes:
\begin{enumerate}[label=\textup{(\roman*)}]
\item the maximum degree is at most $2d$, and, for every $r\geq0$ and every
rooted $r$-ball type $\tau$ that is not a $d$-lattice type, no vertex has
rooted $r$-ball type $\tau$;
\item for each $m\geq1$, there is no connected component of size exactly $m$;
\item for every $r\geq0$, every rooted $r$-ball type $\tau$ realized at a
vertex of an infinite component of a subgraph of $\Z^d$, and every $h\geq1$,
there are at least $h$ distinct vertices with rooted $r$-ball type $\tau$.
\end{enumerate}
Each item is a genuine first-order scheme: bounded degree, a fixed finite
rooted ball type, the existence of a component of a fixed finite size, and
the existence of $h$ distinct realizations of a fixed type are all
first-order expressible.
\end{theorem}

Because $\eta_N\leq1/2$, the hypotheses imply
$f(N)=o(\log N)=o(L)$ and $\log(1/\eta_N)=o(\log N)$.

The following theorem gives the unpruned one-sided polynomial answer in
dimensions at least two.  The one-dimensional case is different because
deleting an edge from a cycle is a cut; see Remark~\ref{rem:one-dimensional}.

\begin{theorem}\label{thm:polynomial}
Fix $d\geq2$, let $N=L^d$, and let $X_L=X_L^d(q_N)$.
\begin{enumerate}[label=\textup{(\alph*)}]
\item Suppose $p_N\to0$, $q_N\to1$, and
$p_N=N^{-\alpha+o(1)}$ for some $\alpha>0$.
      \begin{enumerate}[label=\textup{(\roman*)}]
      \item If $\alpha\notin\{1/k:k\in\N\}$, then $(X_L)$ has a first-order
      zero--one law.
      \item If $\alpha=1/k$ and $\lambda_N=N p_N^k\to0$ or
      $\lambda_N\to\infty$, then $(X_L)$ has a first-order zero--one law.
      \item If $\alpha=1/k$ and
      $\lambda_N=N p_N^k\to\lambda\in(0,\infty)$, then $(X_L)$ has a
      first-order convergence law but not a zero--one law.
      \item If $\alpha=1/k$ and $\lambda_N=N p_N^k$ has no limit in the
      extended interval $[0,\infty]$, then $(X_L)$ does not have a
      first-order convergence law.
      \end{enumerate}
\item The same four alternatives hold with $p_N$ and $q_N$ interchanged.
Thus, if $q_N\to0$, $p_N\to1$, and $q_N=N^{-\alpha+o(1)}$, the critical
quantity at $\alpha=1/k$ is $N q_N^k$.
\end{enumerate}
\end{theorem}

The proof of the monadic theorem and the contrasting torus obstructions are
given in Section~\ref{sec:mso}.  Proposition~\ref{prop:mso-torus-obstruction}
and Corollary~\ref{cor:mso-pruned-torus-obstruction} show that the torus
statements have no direct monadic analogue, even in $\mathrm{MSO}_1$, without
additional global restrictions.

The paper is organized as follows.  Section~\ref{sec:locality} collects the
locality and finite-type tools used throughout the paper: Hanf locality for
bounded-degree graphs, a finite-rank replacement for forests of unbounded
degree, and the elementary counting lemmas for local witnesses and rare defect
patterns.  Section~\ref{sec:ER} proves the pruning theorem for very sparse
Erd\H{o}s--R\'enyi graphs and the star obstruction showing that the smallness
assumption on $f$ cannot be weakened in the most naive way.  Section~\ref{sec:torus}
proves the torus results: first the pruned middle-regime theorem, and then the
classification of the unpruned one-sided polynomial regimes.
Section~\ref{sec:mso} proves the main $\mathrm{MSO}_2$ theorem by component
composition and tree automata, and then gives the $\mathrm{MSO}_1$
bipartiteness obstruction for tori.

\subsection*{Notation and conventions}

All graphs are finite unless explicitly stated otherwise.  The first-order
language of graphs has one binary relation symbol $E$, interpreted as a
symmetric irreflexive adjacency relation.  A sequence of random graphs
$(X_n)$ has a first-order convergence law if $\Prob(X_n\models\varphi)$
converges for every first-order sentence $\varphi$, and it has a first-order
zero--one law if each of these limits is either $0$ or $1$.

For monadic second-order logic we use $\mathrm{MSO}_2$ in the usual
graph-theoretic sense: one may quantify over sets of vertices and over sets
of edges, equivalently over monadic predicates in the incidence structure of
the graph.  The fragment in which second-order variables range only over
subsets of the vertex set is denoted $\mathrm{MSO}_1$.  Monadic convergence
laws and monadic zero--one laws are defined analogously, with $\varphi$ ranging
over sentences of the corresponding monadic logic.

For a graph $G$ and a vertex $v$, $C_G(v)$ denotes the connected component of
$v$.  If $r\geq0$, then $B_G(v,r)$ denotes the induced rooted graph on the
vertices at graph distance at most $r$ from $v$, rooted at $v$.  If $\tau$ is
a rooted $r$-ball type, then $N_\tau(G)$ denotes the number of vertices of
$G$ whose rooted $r$-ball has type $\tau$.

We write $a_n=o(b_n)$ and $a_n=O(b_n)$ in the usual asymptotic sense.  The
phrase ``with high probability'' means with probability tending to one.  All
pruning thresholds are positive integers, and all threshold functions $f$
are positive integer-valued.  For the torus, $L$ is the side length, $d$ is
fixed, and $N=L^d$ is the number of
vertices.  In bond percolation on $T_L^d$, the retained-edge probability is
$q_N$ and the omitted-edge probability is $p_N=1-q_N$.

\section{Locality and finite-type tools}\label{sec:locality}

We first record the bounded-degree locality theorem and the finite-rank
composition facts used later.

\subsection{Hanf locality in bounded degree}

\begin{definition}
Let $\Delta,r,M\in\N$.  Two graphs $G$ and $H$ of maximum degree at most
$\Delta$ are \emph{$(r,M)$-Hanf-equivalent} if, for every rooted $r$-ball
type $\tau$ of degree at most $\Delta$, either
\[
   N_\tau(G)=N_\tau(H)<M,
   \qquad\text{or}\qquad
   N_\tau(G),N_\tau(H)\geq M.
\]
\end{definition}

\begin{theorem}[Bounded-degree Hanf locality]\label{thm:Hanf}
For every $\Delta$ and every first-order quantifier rank $s$, there are
$r=r(\Delta,s)$ and $M=M(\Delta,s)$ such that any two finite or infinite
graphs of maximum degree at most $\Delta$ that are $(r,M)$-Hanf-equivalent
satisfy the same first-order sentences of quantifier rank at most $s$.
\end{theorem}

The general bounded-degree form goes back to Hanf \cite{Hanf}; the finite
form used for the random finite graphs below is due to Fagin, Stockmeyer, and
Vardi \cite{FaginStockmeyerVardi}.  See also \cite{Libkin}.  In the torus
arguments the maximum degree is $2d$.

\subsection{Large-component locality for forests}

The pruned Erd\H{o}s--R\'enyi graph is a forest with high probability,
but its degree is not bounded.  We therefore use a componentwise finite-rank
replacement for Hanf locality.

For finite graphs $A$ and $B$, write $A\equiv_s B$ if they satisfy the same
first-order sentences of quantifier rank at most $s$.  There are only
finitely many $\equiv_s$-classes of finite trees; we call each such class an
\emph{$s$-component type}.  A type is \emph{unbounded} if it contains trees
of arbitrarily large order, and \emph{bounded} otherwise.

\begin{lemma}[Semilinear tree spectra]\label{lem:tree-spectrum}
Let $\mathcal C$ be an $\mathrm{MSO}_2$-definable class of finite
(graph-theoretic) trees.  Then its order spectrum
\[
       S(\mathcal C)=\{|V(T)|:T\in\mathcal C\}
\]
is a semilinear subset of $\N$.

Consequently, if $\mathcal C_1,\ldots,\mathcal C_\nu$ are finitely many such
classes and each has members of arbitrarily large order, then there is a
constant $C\geq1$ such that, for every $i\leq\nu$ and every $m\geq1$, some
$T\in\mathcal C_i$ satisfies
\[
                 m\leq |V(T)|\leq Cm.
\]
\end{lemma}

\begin{proof}
View a graph-theoretic tree through its two-sorted incidence structure, with
unary labels distinguishing vertex-elements from edge-elements.  The
incidence structure is itself a finite labeled tree.  Choose a root, order
the children of each node, and apply the standard first-child/next-sibling
binary encoding.  The class of valid encodings whose decoded incidence tree
belongs to $\mathcal C$ is MSO-definable on finite labeled binary trees:
validity of the encoding, the two sorts, the incidence relation, and the
interpretation of the defining sentence for $\mathcal C$ are all MSO-definable.
By the MSO--tree-automaton theorem, this class is a regular finite-tree
language \cite{ThatcherWright,Doner,Courcelle}.

Let $\mathcal A$ be a finite bottom-up automaton recognizing this language.
Convert $\mathcal A$ into a context-free grammar as follows.  The automaton
states are nonterminals.  A transition with parent state $q$, child states
$q_1,\ldots,q_t$ ($0\leq t\leq2$), and a vertex label gives a production
\[
                 q\longrightarrow xq_1\cdots q_t,
\]
whereas a transition with an edge label gives the same production with $y$
in place of $x$.  A new start symbol $S$ has a production $S\to q$ for each accepting
state $q$.  The Parikh vectors of the generated words are exactly the pairs
$(|V(T)|,|E(T)|)$ arising from trees in $\mathcal C$.  Parikh's theorem
therefore makes this set semilinear, and projection onto the first coordinate
shows that $S(\mathcal C)$ is semilinear \cite{Parikh}.

An infinite semilinear subset of $\N$ contains an infinite arithmetic
progression $u+jh$ with $h>0$.  The least term of this progression that is at
least $m$ is at most $m+u+h$, and hence at most $(1+u+h)m$.  Taking the
maximum of these constants over the finitely many classes proves the final
assertion.
\end{proof}

\begin{lemma}\label{lem:pumping}
For every $s$ there is a constant $C_s$ such that every unbounded
$s$-component type $\gamma$ has the following property: for every $m\geq1$
there is a finite tree $T\in\gamma$ with
\[
       m\leq |T|\leq C_s m.
\]
\end{lemma}

\begin{proof}
Each $s$-component type is defined, on finite trees, by a characteristic
first-order sentence and is therefore $\mathrm{MSO}_2$-definable.  Apply the
second assertion of Lemma~\ref{lem:tree-spectrum} to the finite family of
unbounded $s$-component types.
\end{proof}

\begin{lemma}\label{lem:forest-FV}
For every $s$ there is $M_s$ such that the following holds.  If, for each
$s$-component type $\gamma$, two forests $F,F'$ either have the same number
of components of type $\gamma$, this number being less than $M_s$, or both
have at least $M_s$ components of type $\gamma$, then $F\equiv_s F'$.
\end{lemma}

\begin{proof}
Take $M_s=s+1$.  In the $s$-move Ehrenfeucht--Fra\"isse game on $F$ and
$F'$, whenever Spoiler first enters a component, Duplicator chooses a fresh
component of the same $s$-component type.  At most $s$ components are entered,
so the multiplicity hypothesis always provides such a component.  On every
matched pair, Duplicator follows a winning $s$-move strategy for the two
component trees.  Since distinct components have no edges between them, the
union of the componentwise partial isomorphisms is a partial isomorphism of
the two forests.
\end{proof}

\begin{proposition}\label{prop:forest-profile}
Fix $s$.  There are constants $B_s,M_s$ and a finite set
$\Gamma_s^\infty$ of unbounded $s$-component types such that the following
holds.  If a finite forest $F$ has no component of order at most $B_s$ and has
at least $M_s$ components of each type in $\Gamma_s^\infty$, then the
$s$-theory of $F$ is fixed, independent of $F$.
\end{proposition}

\begin{proof}
Let $\Gamma_s^\infty$ be the set of unbounded $s$-component types, and let
$M_s$ be given by Lemma~\ref{lem:forest-FV}.  Since there are only finitely
many $s$-component types, the bounded ones have a common size bound $B_s$;
if there are no bounded types, take $B_s=0$.  If $F$ has no component of size
at most $B_s$, then no bounded type occurs in $F$.  The unbounded types all
occur at least $M_s$ times by hypothesis.  Lemma~\ref{lem:forest-FV} now
implies that any two such forests are $\equiv_s$.
\end{proof}

\subsection{Counting local witnesses}

\begin{lemma}\label{lem:many}
Let $(A_i)_{i\in I_N}$ be events, each determined by finitely many
independent coin flips.  Suppose that $J_N\subseteq I_N$ is a subfamily for
which the events $(A_i)_{i\in J_N}$ are mutually independent.  If
\[
       |J_N|\min_{i\in J_N}\Prob(A_i)\longrightarrow\infty,
\]
then, for every fixed $m$, at least $m$ of the events $A_i$ occur with
probability tending to one.
\end{lemma}

\begin{proof}
The number of occurrences in the independent subfamily stochastically
dominates a binomial random variable whose mean tends to infinity.  A
Chernoff bound gives the conclusion.
\end{proof}

\begin{lemma}\label{lem:poisson}
Let $\mathcal G$ be a fixed connected locally finite graph, and let a group
$\Gamma\leq\Aut(\mathcal G)$ act transitively on its vertices.  Let
$\pi_N:\mathcal G\to\mathcal G_N$ be finite covering quotients with
$|V(\mathcal G_N)|=N\to\infty$ and injectivity radius tending to infinity;
that is, for every fixed $R$, the quotient map restricts to a graph
isomorphism on every $R$-ball for all sufficiently large $N$.  On the edges
of $\mathcal G_N$,
independently put a rare state with probability $\rho_N$ and the common state
otherwise, where $\rho_N\to0$; write $\mathcal X_N$ for the resulting random
edge-state structure (or for the graph determined by those states).

Fix $k\geq1$ and finitely many $\Gamma$-orbits $D\in\mathcal D$ of bounded
marked configurations in $\mathcal G$.  Choose a common bounded protective
window containing all their decision windows.  A representative of $D$
consists of a specified $k$-element set of rare edges in this window; its
canonical occurrence requires those edges to be rare and every other edge in
the protective window to be common.  Descriptions with the same underlying
rare-edge set are grouped into a single orbit before counting, so distinct
members of $\mathcal D$ have distinct orbits of underlying rare-edge sets.
Assume that the number of candidate occurrences of $D$ in
$\mathcal G_N$ is
\[
             a_D N+o(N)
\]
for a constant $a_D>0$.  Let $W_{D,N}$ count these underlying occurrences,
whose prescribed rare/common states are realized, once per rare-edge set
rather than once per possible rooted observation.

If
\[
        N\rho_N^k\longrightarrow\lambda\in(0,\infty),
\]
then $(W_{D,N})_{D\in\mathcal D}$ converges jointly to a vector of independent
Poisson random variables with respective means $a_D\lambda$.

Suppose, in addition, that $\tau_1,\ldots,\tau_m$ are rooted local types such
that every realization of any $\tau_i$ either uses exactly $k$ rare edges and
is represented in $\mathcal D$, or uses at least $k+1$ rare edges in one of
finitely many bounded decision windows.  Then, for every fixed
truncation level $M$, the vector
\[
 \bigl(\min\{N_{\tau_i}(\mathcal X_N),M\}\bigr)_{1\leq i\leq m}
\]
has a limiting distribution.  The untruncated rooted-type counts need not be
Poisson: one underlying defect configuration may produce a bounded cluster of
rooted observations, so their limits are in general compound Poisson.
\end{lemma}

\begin{proof}
We use mixed factorial moments.  Fix nonnegative integers
$(j_D)_{D\in\mathcal D}$ and put $j=\sum_D j_D$.  Ordered candidate tuples
whose decision windows are pairwise disjoint number
\[
       (1+o(1))\prod_{D\in\mathcal D}(a_D N)^{j_D}.
\]
Their joint occurrence probability is
$(1+o(1))\rho_N^{k j}$, since only a fixed number of common-edge requirements
is involved.  Their contribution to the mixed factorial moment therefore
tends to
\[
       \prod_{D\in\mathcal D}(a_D\lambda)^{j_D}.
\]

It remains to discard overlapping tuples.  Join two proposed occurrences
when their decision windows intersect.  Because the local model and all
windows are fixed, a connected overlap component has only $O(N)$ possible
placements.  A non-singleton component contains two distinct underlying
$k$-edge sets and hence requires at least $k+1$ rare edges.  Its contribution
is therefore $O(N\rho_N^{k+1})=o(1)$; the other overlap components contribute
at most bounded factors because $N\rho_N^k=O(1)$.  Thus every mixed factorial
moment has the limit displayed above, proving the joint independent Poisson
limit.

Finally, the expected number of bounded windows containing a realization
with at least $k+1$ rare edges is $O(N\rho_N^{k+1})=o(1)$.  The same estimate
discards overlaps between distinct critical clusters and additional rare
edges in their decision windows.  Off this exceptional event, each rooted
type count, truncated at $M$, is a fixed function of the finite vector of
underlying cluster counts.  The continuous mapping theorem on the discrete
state space gives the asserted joint limiting distribution.
\end{proof}

\section{Pruning very sparse Erd\H{o}s--R\'enyi graphs}\label{sec:ER}

Throughout this section $p_n=c_n/n$, where $c_n>0$ and $c_n\to0$.
Let $f:\N\to\N$ satisfy $f(n)\to\infty$ and
\[
       f(n)\bigl(\log f(n)+\log(1/c_n)\bigr)=o(\log n).
\]
In particular, $\log(1/c_n)=o(\log n)$ and $f(n)=o(\log n)$.
Write $G_n=G(n,p_n)$ and $H_n=\Prune_{f(n)}(G_n)$.

\begin{lemma}\label{lem:ER-forest}
With probability tending to one, $G_n$ is acyclic.  Consequently, $H_n$ is a
forest with probability tending to one.
\end{lemma}

\begin{proof}
For all sufficiently large $n$, we have $c_n<1$.  The expected number of
cycles in $G_n$ is at most
\[
   \sum_{\ell\geq3} \frac{(n p_n)^\ell}{2\ell}
     =\sum_{\ell\geq3}\frac{c_n^\ell}{2\ell}=o(1).
\]
The result follows from Markov's inequality.
\end{proof}

\begin{proposition}\label{prop:ER-components}
Let $(F_n)$ be any sequence of finite trees with $t_n=|F_n|$ such that
\[
       t_n\bigl(\log t_n+\log(1/c_n)\bigr)=o(\log n).
\]
Let $X_n(F_n)$ be the number of connected components of $G_n$ isomorphic
to $F_n$.  Then
\[
       X_n(F_n)\longrightarrow\infty
       \quad\text{in probability.}
\]
\end{proposition}

\begin{proof}
Since $\log(1/c_n)\to\infty$, the hypothesis implies $t_n=o(\log n)$,
and in particular $t_n=o(n^{1/2})$.  Let $a_n=|\Aut(F_n)|$.  For a fixed
$t_n$-set of vertices, the number of labeled graphs on that set isomorphic
to $F_n$ is $t_n!/a_n$.  Hence
\[
\Ex X_n(F_n)
 =\binom{n}{t_n}\frac{t_n!}{a_n}
   p_n^{t_n-1}(1-p_n)^{t_n(n-t_n)+\binom{t_n}{2}-(t_n-1)}.
\]
Put
\[
       A_n=t_n(n-t_n)+\binom{t_n}{2}-(t_n-1).
\]
Then
\[
 \log\frac{(n)_{t_n}}{n^{t_n}}=O(t_n^2/n)=o(1),
 \qquad
 A_n\log(1-p_n)=-c_n t_n+o(1).
\]
Using $a_n\leq t_n!$ and $c_n t_n=o(\log n)$ therefore gives
\[
\begin{aligned}
\log \Ex X_n(F_n)
&\geq \log n-(t_n-1)\log(1/c_n)-\log(t_n!)-O(c_n t_n)-o(1) \\
&\geq \log n-t_n\bigl(\log(1/c_n)+\log t_n\bigr)-o(\log n) \\
&=(1-o(1))\log n.
\end{aligned}
\]
Thus $\Ex X_n(F_n)\to\infty$.

For completeness, write $X_n(F_n)=\sum_S I_S$, where the sum is over
$t_n$-subsets and $I_S$ is the indicator that $S$ induces a component
isomorphic to $F_n$.  Two distinct intersecting sets cannot both be connected
components.  For two disjoint sets, the joint probability differs from the
product of the individual probabilities only because the absence of the
edges between the two sets is counted twice in the product; the correction
factor is
\[
       (1-p_n)^{-t_n^2}=1+o(1).
\]
Moreover, the proportion of ordered pairs of $t_n$-sets which are disjoint is
\[
        \frac{\binom{n-t_n}{t_n}}{\binom n{t_n}}=1+o(1),
\]
because $t_n^2/n=o(1)$.  It follows that
\[
        \Ex\bigl(X_n(F_n)\bigr)_2
          =(1+o(1))\bigl(\Ex X_n(F_n)\bigr)^2,
\]
and hence
\[
       \Var(X_n(F_n))=o\bigl((\Ex X_n(F_n))^2\bigr)+O(\Ex X_n(F_n)),
\]
and Chebyshev's inequality gives $X_n(F_n)/\Ex X_n(F_n)\to1$ in probability.
\end{proof}

\begin{proof}[Proof of Theorem~\ref{thm:ER-core}]
Fix a first-order sentence $\varphi$ and let $s=\qr(\varphi)$.  Take
$B_s,M_s$, and $\Gamma_s^\infty$ from
Proposition~\ref{prop:forest-profile}, and take $C_s$ from
Lemma~\ref{lem:pumping}.  For every $\gamma\in\Gamma_s^\infty$, that lemma
gives a finite tree $F_{\gamma,n}\in\gamma$ with
\[
       f(n)\leq |F_{\gamma,n}|\leq C_s f(n).
\]
Since $|F_{\gamma,n}|\leq C_s f(n)$, the hypothesis on $f$ implies
\[
\begin{aligned}
 |F_{\gamma,n}|\bigl(\log |F_{\gamma,n}|+\log(1/c_n)\bigr)
 &\leq C_s f(n)\bigl(\log(C_s f(n))+\log(1/c_n)\bigr)\\
 &=o(\log n).
\end{aligned}
\]
By Proposition~\ref{prop:ER-components}, $G_n$ has at least $M_s$ connected
components isomorphic to $F_{\gamma,n}$ with probability tending to one.
Since $\Gamma_s^\infty$ is finite, these events hold simultaneously for all
$\gamma\in\Gamma_s^\infty$ with probability tending to one.  Since
$|F_{\gamma,n}|\geq f(n)$, all of these components survive in $H_n$.

Every component of $H_n$ has order at least $f(n)$, which exceeds $B_s$
for all sufficiently large $n$.  By Lemma~\ref{lem:ER-forest}, $H_n$ is also
a forest with probability tending to one.  Hence $H_n$ satisfies the
hypotheses of Proposition~\ref{prop:forest-profile} with probability tending
to one.  The truth value of every sentence of rank at most $s$, and in
particular of $\varphi$, is then forced.  Since $\varphi$ was arbitrary,
$(H_n)$ has a first-order zero--one law.

For the more precise multiplicity assertion, fix $K$ and repeat the same
argument with $K$ in place of $M_s$ in the application of
Proposition~\ref{prop:ER-components}.  The exclusion of bounded types follows
from $f(n)>B_s$, and the forest assertion follows from
Lemma~\ref{lem:ER-forest}.
\end{proof}

\begin{proposition}\label{prop:star-obstruction}
There are sequences $c_n\to0$ and $f(n)\to\infty$ satisfying
\[
       \log(1/c_n)=o(\log n),
       \qquad
       f(n)\log(1/c_n)=o(\log n),
\]
for which the pruned graphs $\Prune_{f(n)}(G(n,c_n/n))$ do not have a first-order zero--one law.
\end{proposition}

\begin{proof}
Fix $\lambda\in(0,\infty)$ and set
\[
       f(n)=\left\lfloor \frac{\log n}{\log\log n}\right\rfloor.
\]
Choose $c_n\in(0,1)$ so that
\begin{equation}\label{eq:choose-c-star}
       \frac{n c_n^{f(n)-1}e^{-c_n f(n)}}{(f(n)-1)!}=\lambda.
\end{equation}
For all sufficiently large $n$ such a solution exists in the range
$c_n\in(0,e^{-1})$.  Indeed, the left hand side tends to $0$ as $c_n\downarrow0$,
while at $c_n=e^{-1}$ its logarithm is
$\log n-f(n)\log f(n)+O(f(n))\to\infty$.  It is unique in this interval because
the logarithmic derivative of its $c_n$-dependent factor is
\[
        \frac{f(n)-1}{c_n}-f(n)>0
        \qquad (0<c_n<e^{-1})
\]
for all sufficiently large $n$.  Writing
$a_n=\log(1/c_n)$, Stirling's formula applied to
\eqref{eq:choose-c-star} gives
\[
\begin{aligned}
 a_n
 &=\frac{\log n-\log((f(n)-1)!)-c_n f(n)-\log\lambda}{f(n)-1}\\
 &=\log\log\log n+O(1).
\end{aligned}
\]
Hence $c_n\to0$, $\log(1/c_n)=o(\log n)$, and
$f(n)\log(1/c_n)=o(\log n)$.  Notice also that
$f(n)\log f(n)=(1+o(1))\log n$, so the stronger hypothesis of
Theorem~\ref{thm:ER-mso} is deliberately violated.

Let $Z_n$ be the number of connected components of $G(n,c_n/n)$ that are
stars with at least $f(n)$ vertices, and let $\mu_{n,t}$ be the expected number
with exactly $t$ vertices.  For $t\geq3$, direct counting gives
\[
 \mu_{n,t}=\binom nt t\,p_n^{t-1}
 (1-p_n)^{t(n-t)+\binom t2-(t-1)}.
\]
Since $f(n)=o(n^{1/2})$, uniformly for $t=O(f(n))$,
\[
       \mu_{n,t}=(1+o(1))
       \frac{n c_n^{t-1}e^{-c_n t}}{(t-1)!}.
\]
By \eqref{eq:choose-c-star}, $\mu_{n,f(n)}\to\lambda$.  Moreover, for
$f(n)\leq t<n$,
\[
 \frac{\mu_{n,t+1}}{\mu_{n,t}}
 =\frac{(n-t)p_n}{t}(1-p_n)^{n-t-2}
 \leq \frac{c_n}{t(1-p_n)}=o(1)
\]
uniformly in $t$.  Hence the expected number of star components with more
than $f(n)$ vertices is $o(1)$, and $\Ex Z_n\to\lambda$.

Let $Y_n$ count star components with exactly $f(n)$ vertices.  Write it
as a sum over centered candidates $(S,x)$, where $|S|=f(n)$ and $x\in S$ is
the proposed center.  For every fixed $j$, only candidates with pairwise
disjoint vertex sets contribute to the $j$th factorial moment of $Y_n$.
Their joint probability differs from the product of the individual
probabilities by the factor
\[
 (1-p_n)^{-\binom j2 f(n)^2}=1+o(1),
\]
and the use of disjoint vertex sets changes the corresponding combinatorial
factor by $1+o(1)$ because $f(n)=o(n^{1/2})$.  Thus
\[
        \Ex (Y_n)_j\longrightarrow\lambda^j.
\]
The factorial-moment criterion gives $Y_n\Rightarrow\Po(\lambda)$, while
$Z_n-Y_n\to0$ in probability by the preceding tail estimate.  Therefore
\[
       Z_n\Rightarrow \Po(\lambda).
\]

Let $\psi_\star$ be the first-order sentence
\[
 \exists x\Bigl[(\exists y\,E(x,y))\wedge
 \forall y\bigl(E(x,y)\to \forall z(E(y,z)\to z=x)\bigr)\Bigr].
\]
With high probability, $G(n,c_n/n)$ is a forest.  On this event, the
pruned graph satisfies $\psi_\star$ exactly when at least one star component
with at least $f(n)$ vertices survives.  Therefore
\[
       \Prob\bigl(\Prune_{f(n)}(G(n,c_n/n))\models\psi_\star\bigr)
       \longrightarrow 1-e^{-\lambda},
\]
which is strictly between $0$ and $1$.
\end{proof}

\begin{remark}\label{rem:star}
The strengthened smallness condition is not cosmetic.  Let $S_t$ be the
star on $t$ vertices.  When $t=t(n)\to\infty$ and $t=o(n^{1/2})$, the expected
number of components isomorphic to $S_t$ in $G(n,c_n/n)$ is asymptotic to
\[
       n\frac{c_n^{t-1}}{(t-1)!}e^{-c_n t}.
\]
After pruning, the first-order sentence asserting the existence of a vertex
all of whose neighbors are leaves is controlled by star components of order
at least $f(n)$.  The weaker condition
$f(n)\log(1/c_n)=o(\log n)$ permits choices of $c_n$ and $f(n)$ for which
the displayed expectation at $t=f(n)$ has a positive finite limit or
oscillates.  The corresponding sentence then need not have a $0$--$1$ limit.
The additional condition $f(n)\log f(n)=o(\log n)$ excludes this
automorphism-driven obstruction.
\end{remark}

\section{Percolated tori}\label{sec:torus}

Fix $d\geq1$ and write $N=L^d$.  For every fixed coordinate box, the
quotient map $\Z^d\to T_L^d$ is injective on the box and its incident edges
once $L$ is sufficiently large.

A rooted $r$-ball type $\tau$ is a \emph{$d$-lattice type} if
$B_X(0,r)\cong\tau$ for some subgraph $X\subseteq\Z^d$.  It is
\emph{unbounded} if such a realization can be chosen with the root in an
infinite component of $X$.  This is stronger than merely requiring the root
component in the displayed ball to reach distance $r$.

\subsection{The pruned middle regime}

Throughout this subsection $f:\N\to\N$ is integer-valued.

Let
\[
       X_L=X_L^d(q_N),\qquad Y_L=\Prune_{f(N)}(X_L).
\]
Recall that
\[
       \eta_N=\min\{q_N,p_N\},\qquad p_N=1-q_N.
\]

\begin{lemma}\label{lem:torus-witness}
Assume $f(N)\to\infty$ and
\[
       f(N)\log(1/\eta_N)=o(\log N).
\]
Fix $r\geq0$, an unbounded $d$-lattice rooted $r$-ball type $\tau$, and
$m\geq1$.  With probability tending to one, $Y_L$ has at least $m$ vertices
whose rooted $r$-ball is $\tau$.
\end{lemma}

\begin{proof}
Choose a subgraph $X\subseteq\Z^d$ such that $B_X(0,r)$ has type $\tau$ and
$0$ lies in an infinite component.  That locally finite component contains a
self-avoiding ray $\gamma$ starting at $0$.  There is a fixed finite set of
lattice edges $D_\tau$, contained in a box depending only on $r$, $d$, and
$\tau$, whose states determine the rooted
$r$-ball.  For example, one may take all lattice edges with both endpoints in
$[-r-1,r+1]^d$.  Prescribe on $D_\tau$ the same states as in $X$.  This includes every boundary decision
needed to prevent an additional vertex from entering the $r$-ball.

Now take an initial segment of $\gamma$ containing at least $f(N)$ distinct
vertices and prescribe all of its edges to be retained.  These prescriptions
are consistent, since $\gamma\subseteq X$.  They force the root component to
have at least $f(N)$ vertices while leaving its rooted $r$-ball equal to
$\tau$.  The number of queried edges is at most
$C_{r,d,\tau}(f(N)+1)$, and their coordinate diameter is
$O_{r,d,\tau}(f(N))$.  Because $\eta_N\leq1/2$, the hypothesis implies
$f(N)=o(\log N)$ and therefore $f(N)=o(L)$.  Thus the witness embeds without
wrap-around in $T_L^d$ for all sufficiently large $L$.

A specified translate of this witness occurs with probability at least
\[
       \eta_N^{C_{r,d,\tau}(f(N)+1)}
       =\exp\{-C_{r,d,\tau}(f(N)+1)\log(1/\eta_N)\}
       =N^{-o(1)}.
\]
The hypothesis also gives $\log(1/\eta_N)=o(\log N)$, so the additive
$1$ in the exponent is harmless.  Moreover,
$f(N)=o(\log N)=N^{o(1)}$.  A greedy packing of translates of the witness
box gives at least
\[
       c_{r,d,\tau}\frac{N}{(f(N)+1)^d}=N^{1-o(1)}
\]
translates with pairwise disjoint queried edge sets.  The product of this
quantity and the displayed lower bound tends to infinity, so
Lemma~\ref{lem:many} gives at least $m$ occurrences with probability tending
to one.  In each occurrence, the root lies in an $X_L$-component of order at
least $f(N)$.  Pruning retains that entire component, so the rooted $r$-ball
remains $\tau$ in $Y_L$.
\end{proof}

\begin{lemma}\label{lem:torus-no-bounded}
For each fixed $r$, for all sufficiently large $L$, no vertex of $Y_L$ has a
rooted $r$-ball type that is not unbounded.
\end{lemma}

\begin{proof}
Fix $r$.  For all sufficiently large $L$, the coordinate box
\[
        Q_r(v)=v+[-r,r]^d
\]
embeds in $T_L^d$ without wrap-around, and every graph-metric $r$-ball around
$v$ is contained in $Q_r(v)$.

Suppose that the component of $v$ in $Y_L$ contains a vertex outside
$Q_r(v)$.  Choose a shortest component path from $v$ to the first vertex $x$
outside this box and lift the path to $\Z^d$.  One coordinate of $x-v$ has
absolute value $r+1$.  Starting at $x$, continue indefinitely in the
corresponding outward coordinate direction.  Form a subgraph of $\Z^d$ by
taking the lifted rooted ball $B_{Y_L}(v,r)$, the part of the chosen
path beyond that ball, and this outward ray, while omitting all other edges.
Every
new vertex is at graph distance greater than $r$ from the root, so the rooted
$r$-ball is unchanged, and the root now lies in an infinite component.
Hence the rooted type of $v$ is unbounded.

It follows that a vertex whose rooted $r$-ball type is not unbounded must
have
its entire component contained in $Q_r(v)$.  Such a component has at most
$(2r+1)^d$ vertices.  Since $f(N)\to\infty$, this is smaller than $f(N)$ for
all sufficiently large $L$, contradicting the definition of
$Y_L=\Prune_{f(N)}(X_L)$.
\end{proof}

\begin{proof}[Proof of Theorem~\ref{thm:torus-core}]
Fix a first-order sentence $\varphi$, let $s=\qr(\varphi)$, and take $r,M$
from Theorem~\ref{thm:Hanf} for maximum degree $2d$ and rank $s$.  For all
large $L$, every rooted $r$-ball in $T_L^d$ lifts to $\Z^d$ and hence is a
$d$-lattice type; all non-lattice types are therefore deterministically
absent.  There are only finitely many rooted $r$-ball types of degree at most
$2d$.  Lemma~\ref{lem:torus-witness}, followed by a union bound over the
unbounded ones, shows that each unbounded $d$-lattice type occurs at least
$M$ times in $Y_L$ with probability tending to one.
Lemma~\ref{lem:torus-no-bounded} excludes every remaining type.  Thus, with
probability tending to one, the $(r,M)$-Hanf profile is the deterministic
profile
\[
  N_\tau(Y_L)\geq M\quad\Longleftrightarrow\quad
  \tau\text{ is an unbounded $d$-lattice type}.
\]
Hanf locality makes the truth value of $\varphi$ constant on this event.
This proves the zero--one law, including the case in which the radius supplied
by Hanf is $r=0$.

It remains to justify the asserted axiomatization.  The degree clause in
scheme~\textup{(i)} holds deterministically, and its remaining clauses hold
eventually at every fixed radius by the lifting observation above.
Scheme~\textup{(ii)} holds eventually for each fixed $m$ because every
component of $Y_L$ has at least $f(N)>m$ vertices.  Scheme~\textup{(iii)} is
Lemma~\ref{lem:torus-witness}, applied with an arbitrary fixed multiplicity
$h$.  Hence every displayed axiom belongs to the limiting almost-sure theory.
Moreover, every finite subset of the three schemes holds in $Y_L$ with
probability tending to one and is therefore satisfiable.  By compactness, the
full collection of schemes has a model.

Conversely, let $A$ be any graph satisfying all three schemes, fix $v\in A$,
and fix $r\geq0$.  Scheme~\textup{(ii)} implies that the component of $v$ is
infinite.  For every $R\geq r$, scheme~\textup{(i)} gives a rooted lattice
realization of $B_A(v,R)$.  The possible injective root-preserving graph
embeddings of this finite ball into $[-R,R]^d$ form a finite set; the image is
contained in that box because lattice displacement from the root is at most
graph distance.  Join an embedding at level $R+1$ to its restriction at
level $R$.  The resulting tree is infinite and finitely branching, so
K\H{o}nig's
lemma supplies a coherent branch.  The union of that branch embeds the entire
component of $v$ as an infinite subgraph of $\Z^d$.  Consequently
$B_A(v,r)$ is an unbounded $d$-lattice type.  Scheme~\textup{(iii)} now says
that every such type has arbitrarily many realizations.  Thus every model of
the schemes has, at every radius and cutoff, exactly the deterministic Hanf
profile displayed above.  Hanf locality shows that the schemes determine a
complete first-order theory, namely the limiting theory of $(Y_L)$.
\end{proof}

\subsection{One-sided polynomial regimes without pruning}

Throughout this subsection $d\geq2$.  We prove
Theorem~\ref{thm:polynomial}.  We give the details when $p_N\to0$ and
$q_N\to1$; the regime $q_N\to0$ follows by interchanging retained and
omitted edges.

Fix $r\geq0$ and let $W_r$ be the set of lattice edges with both endpoints in
$[-r-1,r+1]^d$.  A path of length at most $r$ from the root stays inside
$[-r,r]^d$, so, for every vertex $v$ and all sufficiently large $L$, the edge
states in the translate $v+W_r$ determine $B_{X_L}(v,r)$.  For a rooted
$d$-lattice $r$-ball type $\tau$, let $b(\tau)$ be the minimum number of
omitted edges among the complete assignments on $W_r$ that produce $\tau$.
Equivalently, an assignment attaining the minimum omits $b(\tau)$ edges and
retains every other edge in $W_r$.  This definition is independent of
enlarging the decision window by finitely many irrelevant edges.  In
particular, $b(\tau)=0$ only for the full rooted $r$-ball of $\Z^d$.

\begin{lemma}\label{lem:defect-expansion}
For every fixed $r$ and every rooted $d$-lattice $r$-ball type $\tau$, there
is a constant $C_\tau>0$ such that, uniformly in the root vertex and for all
sufficiently large $L$,
\begin{equation}\label{eq:type-prob}
 \Prob\bigl(B_{X_L}(v,r)\cong\tau\bigr)
   =C_\tau p_N^{b(\tau)}+O\bigl(p_N^{b(\tau)+1}\bigr).
\end{equation}
The constant depends only on $r$, $d$, and $\tau$.
\end{lemma}

\begin{proof}
The decision window contains only finitely many edges, and the relevant list
of assignments is independent of $v$ and $L$ once the window lifts to
$\Z^d$.  An assignment with $a$ omitted and $c$ retained edges has probability
$p_N^a(1-p_N)^c$.  Summing the probabilities of the mutually exclusive assignments that
realize $\tau$, and grouping them by $a$, gives a polynomial in $p_N$.  Its
least nonzero degree is $b(\tau)$, and its leading coefficient $C_\tau$ is the positive
number of minimal assignments.  This proves the uniform expansion.
\end{proof}

\begin{lemma}\label{lem:poly-zero-one}
Assume $d\geq2$, $p_N\to0$, $q_N\to1$, and
$p_N=N^{-\alpha+o(1)}$ for some $\alpha>0$.
\begin{enumerate}[label=\textup{(\roman*)}]
\item If $\alpha\notin\{1/k:k\in\N\}$, then $(X_L)$ has a first-order
zero--one law.
\item If $\alpha=1/k$ and $N p_N^k\to0$ or $N p_N^k\to\infty$, then
$(X_L)$ has a first-order zero--one law.
\end{enumerate}
\end{lemma}

\begin{proof}
Fix a sentence $\varphi$ and choose $r,M$ from Hanf locality for maximum
degree $2d$ and quantifier rank $\qr(\varphi)$.  Once $L$ is large enough,
every rooted $r$-ball in $X_L$ is a $d$-lattice type, so every non-lattice
type is deterministically absent.  For a $d$-lattice type $\tau$,
equation~\eqref{eq:type-prob} gives
\[
       \Ex N_\tau(X_L)=N C_\tau p_N^{b(\tau)}(1+o(1)).
\]
If $b(\tau)=0$, choose a positive-density set of centers with disjoint
decision windows.  The corresponding independent events have probabilities
tending to $C_\tau>0$, so the type occurs at least $M$ times with high
probability.

For $b(\tau)\geq1$ and $\alpha\notin\{1/j:j\in\N\}$,
\[
       N p_N^{b(\tau)}=N^{1-\alpha b(\tau)+o(1)}
\]
tends either to $0$ or to $\infty$.  If $\alpha=1/k$, the same conclusion
holds for every $b(\tau)\neq k$; for $b(\tau)=k$, it follows from the
additional assumption on $N p_N^k$.

If $N p_N^{b(\tau)}\to0$, Markov's inequality shows that $\tau$ is absent
with high probability.  If $N p_N^{b(\tau)}\to\infty$, a fixed-window packing
provides $c_{r,d}N$ centers with pairwise disjoint decision windows.  The
corresponding events are independent and have probability asymptotic to
$C_\tau p_N^{b(\tau)}$, so Lemma~\ref{lem:many} shows that $\tau$ occurs at
least $M$ times with high probability.

There are only finitely many rooted $r$-ball types of degree at most $2d$.
Taking the intersection of the preceding high-probability events gives a
deterministic truncated Hanf profile.  Hanf locality fixes the truth value of
$\varphi$ with probability tending to one.
\end{proof}

\begin{lemma}\label{lem:poly-convergence}
Assume $d\geq2$, $p_N\to0$, $q_N\to1$, and
$N p_N^k\to\lambda\in(0,\infty)$ for some $k\in\N$.  Then $(X_L)$ has a
first-order convergence law.
\end{lemma}

\begin{proof}
Fix $\varphi$ and choose $r,M$ from Hanf locality for maximum degree $2d$
and quantifier rank $\qr(\varphi)$.  Non-$d$-lattice rooted $r$-ball types
are deterministically absent for all large $L$.  For a $d$-lattice type
$\tau$, Lemma~\ref{lem:defect-expansion} gives \eqref{eq:type-prob}.  If
$b(\tau)<k$, then $N p_N^{b(\tau)}\to\infty$, and the disjoint-window
argument from Lemma~\ref{lem:poly-zero-one} shows that $\tau$ occurs at least
$M$ times with high probability.  If $b(\tau)>k$, then
$N p_N^{b(\tau)}\to0$, and Markov's inequality makes $\tau$ absent with high
probability.  Thus only the types with $b(\tau)=k$ have a non-deterministic
truncated count.

Take a common fixed decision window for all critical types.  List, up to
translation in $\Z^d$, all underlying $k$-edge omitted sets arising in their
minimal assignments, merging all rooted descriptions generated by the same
set.  This is a finite list.  Every listed orbit has $a_D N+o(N)$ candidates
in $T_L^d$, with $a_D>0$, and the quotient maps
$\Z^d\to T_L^d$ have injectivity radius tending to infinity.  Hence
Lemma~\ref{lem:poisson}, with $\rho_N=p_N$, gives a joint Poisson limit for
the underlying clusters.  A realization of a critical type not represented
by a minimal assignment uses at least $k+1$ omitted edges in a fixed window,
and the expected number of all such realizations is
\[
             O(N p_N^{k+1})=O((N p_N^k)p_N)=o(1).
\]
It follows that the vector of critical rooted-type counts, truncated at $M$,
has a joint limiting distribution.  The untruncated counts may be compound
Poisson, because one defect set can be observed from several roots.

With probability tending to one all noncritical coordinates of the Hanf
profile have their deterministic absent or saturated values.  The complete
truncated profile takes values in a finite set, and its distribution
converges.  Hanf locality makes the truth of $\varphi$ a function of that
profile, so $\Prob(X_L\models\varphi)$ converges.
\end{proof}

\begin{lemma}\label{lem:visible-critical-pattern}
For every $d\geq2$ and every $k\geq1$ there are a rooted ball type
$\sigma_{d,k}$, of some fixed radius and with minimal omitted-edge defect
exactly $k$, and a constant $a_{d,k}>0$ with the following property.  Let
$\theta_{d,k}$ be the first-order sentence asserting that some vertex has
rooted type $\sigma_{d,k}$.  In the regime $p_N\to0$, $q_N\to1$, whenever
\[
        N p_N^k\longrightarrow\lambda\in[0,\infty],
\]
one has
\[
        \Prob(X_L\models\theta_{d,k})
        \longrightarrow 1-e^{-a_{d,k}\lambda},
\]
where $e^{-a_{d,k}\infty}$ is interpreted as $0$.
\end{lemma}

\begin{proof}
Choose a root at $0$ and a set $S$ of $k$ pairwise vertex-disjoint lattice
edges in a fixed box, none incident with the root.  Take the edges sufficiently
separated and parallel to the first coordinate axis.  Since $d\geq2$, choose them so that each has a three-edge detour in the
second coordinate direction and these detours avoid $S$.  Consequently the
subgraph
\[
                 X^*=\Z^d\setminus S
\]
obtained by omitting precisely the edges in $S$ is connected: any lattice path
can have each use of an edge in $S$ replaced by its fixed detour.

Choose $R$ so large that every endpoint of an edge in $S$, together with all
of its lattice neighbors, has $X^*$-distance strictly less than $R$ from the
root.  Let
\[
                 \sigma_{d,k}=B_{X^*}(0,R).
\]
Finally let $P$ be the complete assignment on the decision window $W_R$ which
omits the edges in $S$ and retains every edge in $W_R\setminus S$.  The order
of these choices is important: $W_R$ is chosen after the observation radius,
and it contains every edge whose state can affect the rooted $R$-ball.
Therefore an occurrence of $P$ forces rooted type $\sigma_{d,k}$, regardless
of the states outside the translated window.

The displayed realization uses $k$ omitted edges.  In $\sigma_{d,k}$ the
$2k$ endpoints of $S$ are strictly inside the ball and have degree $2d-1$;
here ``interior'' means graph distance less than $R$ from the root.  Every
other interior vertex has degree $2d$.  In any realization of the same
rooted type in a sufficiently large torus, an interior vertex has all of its
retained neighbors inside the ball.  Thus each of these $2k$ vertices must
be incident with an omitted edge.  One omitted edge accounts for at most two
of the deficits, so every realization requires at least $k$ omitted edges.
Hence $b(\sigma_{d,k})=k$.

Assume first that $N p_N^k\to\lambda\in(0,\infty)$.  In the fixed window
$W_R$, list, up to lattice translation, all underlying minimal $k$-edge sets
that realize $\sigma_{d,k}$, forgetting the choice of an observing root and
merging descriptions with the same rare-edge set.  This list is finite and
nonempty.  Lemma~\ref{lem:poisson}, applied to the fixed local model
$\Z^d$ and its torus quotients, shows that their underlying cluster counts
converge jointly to independent Poisson random variables.  Write $a_{d,k}>0$
for the sum of their orbit densities.  The expected number of realizations
using at least $k+1$ omitted edges is
\[
        O(N p_N^{k+1})=O((N p_N^k)p_N)\longrightarrow0.
\]
Markov's inequality therefore shows that, up to an event of probability
$o(1)$, the sentence $\theta_{d,k}$ holds exactly when at least one minimal
defect cluster occurs.  Hence
\[
        \Prob(X_L\models\theta_{d,k})
        \longrightarrow 1-e^{-a_{d,k}\lambda}.
\]

If $N p_N^k\to0$, Lemma~\ref{lem:defect-expansion} gives
$\Ex N_{\sigma_{d,k}}(X_L)=O(N p_N^k)=o(1)$, so Markov's inequality gives the
same formula with $\lambda=0$.  If $N p_N^k\to\infty$, choose a
positive-density family of translates of the particular specification $P$
with pairwise disjoint translates of $W_R$.  Each occurs with probability
\[
       p_N^k q_N^{|W_R|-k}=(1+o(1))p_N^k,
\]
and every occurrence forces rooted type $\sigma_{d,k}$.  Lemma~\ref{lem:many}
shows that at least one occurs with probability tending to one.  This is the
asserted formula for $\lambda=\infty$.
\end{proof}

\begin{lemma}\label{lem:poly-not-zero-one}
Fix $d\geq2$ and $k\geq1$, and assume $p_N\to0$ and $q_N\to1$.
\begin{enumerate}[label=\textup{(\roman*)}]
\item If $N p_N^k\to\lambda\in(0,\infty)$, then the zero--one law fails.
\item If $N p_N^k$ has no limit in $[0,\infty]$, then the convergence law
fails.
\end{enumerate}
\end{lemma}

\begin{proof}
Use the sentence $\theta_{d,k}$ from
Lemma~\ref{lem:visible-critical-pattern}.

If $N p_N^k\to\lambda\in(0,\infty)$, then
\[
       \Prob(X_L\models\theta_{d,k})\longrightarrow 1-e^{-a_{d,k}\lambda},
\]
which is strictly between $0$ and $1$.  Hence the zero--one law fails.

If $N p_N^k$ has no extended limit, compactness of $[0,\infty]$ gives two
distinct subsequential limits $\lambda_1$ and $\lambda_2$.  Along a
subsequence on which $N p_N^k\to\lambda_i$,
Lemma~\ref{lem:visible-critical-pattern} gives
\[
       \Prob(X_L\models\theta_{d,k})\longrightarrow 1-e^{-a_{d,k}\lambda_i},
\]
where $e^{-a_{d,k}\infty}$ is interpreted as $0$.  The map
$\lambda\mapsto1-e^{-a_{d,k}\lambda}$ is strictly increasing on the extended
half-line, so these two subsequential limits are different.  Hence the
convergence law fails.
\end{proof}

\begin{proof}[Proof of Theorem~\ref{thm:polynomial}]
Part~\textup{(a)} follows from Lemmas~\ref{lem:poly-zero-one},
\ref{lem:poly-convergence}, and~\ref{lem:poly-not-zero-one}.  For
part~\textup{(b)}, replace ``omitted'' by ``retained'' throughout.  A rooted
type is then measured by the least number $a(\tau)$ of retained edges needed
to realize it when all nearby edges are omitted, and the same estimates use
$N q_N^{a(\tau)}$.  For the analogue of
Lemma~\ref{lem:visible-critical-pattern}, take a fixed straight lattice path
of length $k$, root it at one endpoint, retain its $k$ edges, and omit every
other lattice edge incident with one of its $k+1$ vertices.  A fixed decision
window then forces the root component to be exactly this path.  Observe the
component at radius $k+1$, so every path vertex is strictly inside the rooted
ball and no retained edge can leave it.  Any realization of this rooted type
therefore contains a connected graph on $k+1$ vertices and hence at least
$k$ retained edges.  The prescribed path uses exactly $k$, so its minimal
retained-edge defect is $k$.  The same
fixed-local-model Poisson argument gives the critical existence probabilities
and completes the proof.
\end{proof}

\begin{remark}\label{rem:one-dimensional}
The polynomial theorem was stated for $d\geq2$ because in dimension one
deleted edges are cuts.  When $d=1$ and $L\geq3$, $T_L^1$ is a cycle and every
component of a proper percolated subgraph is a path.  If $p=1-q$ and
$P_{\ell,L}$ denotes the number of path components with $\ell$ vertices, then
for $1\leq\ell<L$ the exact expectation is
\[
        \Ex P_{\ell,L}=L p^2 q^{\ell-1}.
\]
For $\ell=L$, a path on all $L$ vertices occurs exactly when one cycle edge is
omitted, and hence
\[
        \Ex P_{L,L}=L p q^{L-1}.
\]
Thus the first formula is not uniform up to $\ell=L$.  These path-component
thresholds are one-dimensional phenomena; in dimension $d\geq2$, components
of a percolated torus need not be paths.
\end{remark}

\section{Monadic second-order logic}\label{sec:mso}

We now prove the main monadic theorem.  We work with the standard
incidence presentation of $\mathrm{MSO}_2$: the structure has vertex and edge
sorts, together with the incidence relation, and monadic variables may range
over either sort.  In $\mathrm{MSO}_1$, monadic variables range only over
vertex sets.

The first-order torus arguments rely on locality, which does not extend to
monadic logic in this generality because MSO can express global properties
such as bipartiteness.  The pruned Erd\H{o}s--R\'enyi graph, however, is a
forest with high probability.  On disjoint unions, the Feferman--Vaught
composition method applies to MSO \cite{FefermanVaught,Makowsky}; on each tree
component, MSO is controlled by finite tree automata
\cite{ThatcherWright,Doner,Courcelle}.  We isolate the finite-rank consequence
needed for forests.

\begin{lemma}\label{lem:mso-FV-forest}
For every $s$ there are a finite set $\mathcal A_s$, an
$\mathrm{MSO}_2$-definable partition of the finite trees
\[
        \tau_s:\{\text{finite trees}\}\longrightarrow\mathcal A_s,
\]
and a number $M_s$ with the following property.  For a finite forest $F$ and
$a\in\mathcal A_s$, let
\[
        n_a(F)=
        \bigl|\{C:C \text{ is a component of }F\text{ and }\tau_s(C)=a\}\bigr|.
\]
If two finite forests $F$ and $F'$ satisfy, for every $a\in\mathcal A_s$,
\[
        n_a(F)=n_a(F')<M_s
        \qquad\text{or}\qquad
        n_a(F),n_a(F')\geq M_s,
\]
then $F$ and $F'$ satisfy the same $\mathrm{MSO}_2$ sentences of quantifier
rank at most $s$.  Here ``definable partition'' means that for each
$a\in\mathcal A_s$ there is an $\mathrm{MSO}_2$ sentence $\chi_a$ such that
$\tau_s(T)=a$ if and only if $T\models\chi_a$.
\end{lemma}

\begin{proof}
Represent each graph by its two-sorted incidence structure, with unary
predicates for the vertex and edge sorts.  The finite-rank MSO
Feferman--Vaught theorem for disjoint sums gives, for every $s$, a finite
MSO-definable partition $\mathcal A_s$ of the possible summands and an
integer $r_s$ such that the rank-$s$ MSO theory of a disjoint sum is
determined by the rank-$r_s$ MSO theory of its index structure
\[
       I_F=\bigl(\operatorname{Comp}(F);(P_a)_{a\in\mathcal A_s}\bigr),
\]
where $P_a(C)$ holds exactly when the component $C$ belongs to the class
$a$ \cite{FefermanVaught,Makowsky}.  Taking Boolean cells of the finitely many
component formulas supplied by the composition theorem gives the asserted
partition; each cell is defined by a finite Boolean combination of
$\mathrm{MSO}_2$ sentences and hence by a single $\mathrm{MSO}_2$ sentence
$\chi_a$.

It remains to analyze the index structures.  A pure set with finitely many
unary predicates has only finitely many Boolean cells.  For this fixed finite
unary signature and every quantifier rank $r$, there is a cutoff $M(r)$ such
that two such structures are
MSO-equivalent up to rank $r$ whenever, in each cell, the two cardinalities
are equal below $M(r)$ or are both at least $M(r)$.  This follows directly
from the monadic Ehrenfeucht--Fra\"iss\'e game.  An element move consumes at
most one point of one cell.  A set move splits every current cell into two;
choosing the preceding cutoff at least twice the next cutoff lets Duplicator
match every small part exactly and keep every remaining corresponding part
above the next cutoff.  Induction on the number of remaining moves proves the
claim (and simultaneously handles the finitely many pebbled elements).

Apply this cutoff statement with rank $r_s$ and put $M_s=M(r_s)$.  The
hypothesis of the lemma makes $I_F$ and $I_{F'}$ equivalent up to rank $r_s$;
the composition theorem then makes $F$ and $F'$ equivalent up to rank $s$ in
$\mathrm{MSO}_2$.
\end{proof}

\begin{lemma}\label{lem:mso-tree-pumping}
For every $s$ there is a constant $C_s$ such that the following holds.  If
$a\in\mathcal A_s$ has representatives of arbitrarily large order, then for
every $m\geq1$ there is a finite tree $T$ such that
\[
        \tau_s(T)=a
        \qquad\text{and}\qquad
        m\leq |T|\leq C_s m.
\]
\end{lemma}

\begin{proof}
By Lemma~\ref{lem:mso-FV-forest}, each class
\[
        \mathcal T_a=\{T:\tau_s(T)=a\}
\]
is $\mathrm{MSO}_2$-definable.  Apply the second assertion of
Lemma~\ref{lem:tree-spectrum} to the finite family of classes
$\mathcal T_a$ having representatives of arbitrarily large order.
\end{proof}

\begin{proposition}\label{prop:mso-forest-profile}
Fix $s$.  There are constants $B_s,M_s$ and a finite set
$\mathcal A_s^\infty\subseteq\mathcal A_s$ such that the following holds.  If
a finite forest $F$ has no component of order at most $B_s$ and has at least
$M_s$ components of each state in $\mathcal A_s^\infty$, then the
$\mathrm{MSO}_2$ theory of $F$ up to quantifier rank $s$ is fixed, independent
of $F$.
\end{proposition}

\begin{proof}
Let $\mathcal A_s^\infty$ be the set of states $a\in\mathcal A_s$ which have
finite-tree representatives of arbitrarily large order.  Since $\mathcal A_s$
is finite, all remaining states have representatives of order bounded by a
common constant; call this bound $B_s$, and take $B_s=0$ if there are no
bounded states.  If $F$ has no component of size at
most $B_s$, then no bounded state occurs in $F$.  If, in addition, every state
in $\mathcal A_s^\infty$ occurs at least $M_s$ times, then
Lemma~\ref{lem:mso-FV-forest} shows that the $\mathrm{MSO}_2$ theory of $F$
up to rank $s$ is determined.
\end{proof}

\begin{proof}[Proof of Theorem~\ref{thm:ER-mso}]
Fix an $\mathrm{MSO}_2$ sentence $\varphi$ and let $s$ be its quantifier rank.
Take the constants and states from Proposition~\ref{prop:mso-forest-profile}.
For every $a\in\mathcal A_s^\infty$, Lemma~\ref{lem:mso-tree-pumping} gives a
finite tree $F_{a,n}$ such that
\[
        \tau_s(F_{a,n})=a
        \qquad\text{and}\qquad
        f(n)\leq |F_{a,n}|\leq C_s f(n).
\]
Since $|F_{a,n}|\leq C_s f(n)$, the hypothesis on $f$ gives
\[
\begin{aligned}
 |F_{a,n}|\bigl(\log |F_{a,n}|+\log(1/c_n)\bigr)
 &\leq C_s f(n)\bigl(\log(C_s f(n))+\log(1/c_n)\bigr)\\
 &=o(\log n).
\end{aligned}
\]
By Proposition~\ref{prop:ER-components}, $G(n,p_n)$ has at least $M_s$
connected components isomorphic to $F_{a,n}$ with probability tending to one.
There are only finitely many states $a\in\mathcal A_s^\infty$, so this holds
simultaneously for all of them with probability tending to one.  Since
$|F_{a,n}|\geq f(n)$, all of these components survive in $H_n$.

Every component of $H_n$ has order at least $f(n)>B_s$ for all
sufficiently large $n$.  By Lemma~\ref{lem:ER-forest}, $H_n$ is also a forest
with probability tending to one.  Hence the hypotheses of
Proposition~\ref{prop:mso-forest-profile} hold with probability tending to
one, and the truth value of $\varphi$ is forced on that event.  Thus $(H_n)$
has an $\mathrm{MSO}_2$ zero--one law.
\end{proof}

The preceding theorem is special to the pruned Erd\H{o}s--R\'enyi forests.
There is no corresponding monadic extension of the torus results in the same
generality.  The obstruction is already present in the weaker
$\mathrm{MSO}_1$ fragment: monadic second-order logic can express
bipartiteness, and the parity of the side length $L$ is visible through this
property.

\begin{proposition}\label{prop:mso-torus-obstruction}
Fix $d\geq1$, let $N=L^d$, and let $X_L=X_L^d(q_N)$ with
$q_N=1-p_N$.  Suppose
\[
        L p_N\longrightarrow0.
\]
If $L$ ranges over all positive integers, then $(X_L)$ does not have an
$\mathrm{MSO}_1$ convergence law, and hence does not have an
$\mathrm{MSO}_2$ convergence law.
\end{proposition}

\begin{proof}
In the graph vocabulary, the sentence
\[
 \beta :=
 \exists U\,\forall x\,\forall y\,
 \Bigl(E(x,y)\rightarrow
 \bigl((U(x)\wedge\neg U(y))\vee(\neg U(x)\wedge U(y))\bigr)\Bigr)
\]
says that the graph is bipartite.  This is an $\mathrm{MSO}_1$ sentence, and
therefore also an $\mathrm{MSO}_2$ sentence.  In the two-sorted incidence
presentation, $x,y$ are vertex-sort variables, $U$ is a vertex-set variable,
and $E(x,y)$ is replaced by
\[
 x\neq y\ \wedge\
 \exists e\,\bigl(\operatorname{Inc}(e,x)\wedge
                         \operatorname{Inc}(e,y)\bigr),
\]
with $e$ of the edge sort.  Thus the displayed property has exactly the same
meaning in the standard incidence presentation and uses no edge-set
quantification.

If $L$ is even, then $T_L^d$ is bipartite, using the parity of the sum of the
coordinates.  Hence every subgraph of $T_L^d$ is bipartite, and therefore
\[
        \Prob(X_L\models\beta)=1
\]
for every even $L$.

If $L$ is odd, fix the last $d-1$ coordinates and vary the first coordinate.
This gives a coordinate cycle of length $L$ in $T_L^d$.  It is an odd cycle.
The probability that all its edges are retained is
\[
        q_N^L=(1-p_N)^L\longrightarrow1,
\]
because
\[
        0\leq 1-q_N^L\leq L p_N\longrightarrow0.
\]
Thus, along odd $L$, the graph $X_L$ contains an odd cycle with probability
tending to one.  Hence
\[
        \Prob(X_L\models\beta)\longrightarrow0
\]
along odd $L$.

The same $\mathrm{MSO}_1$ sentence has subsequential limiting probabilities
$1$ and $0$.  Therefore $(X_L)$ has no $\mathrm{MSO}_1$ convergence law, and
a fortiori no $\mathrm{MSO}_2$ convergence law.
\end{proof}

\begin{corollary}\label{cor:mso-pruned-torus-obstruction}
Under the assumptions of Proposition~\ref{prop:mso-torus-obstruction}, suppose
also that $f:\N\to\N$ satisfies $f(N)=o(L)$, and put
\[
        Y_L=\Prune_{f(N)}(X_L).
\]
If $L$ ranges over all positive integers, then $(Y_L)$ does not have an
$\mathrm{MSO}_1$ convergence law, and hence does not have an
$\mathrm{MSO}_2$ convergence law.
\end{corollary}

\begin{proof}
For even $L$, the graph $T_L^d$ is bipartite, and therefore every induced
subgraph of every subgraph of $T_L^d$ is bipartite.  Hence
\[
        \Prob(Y_L\models\beta)=1
\]
for every even $L$, where $\beta$ is the sentence from the proof of
Proposition~\ref{prop:mso-torus-obstruction}.

For odd $L$, the proof of Proposition~\ref{prop:mso-torus-obstruction} shows
that, with probability tending to one, $X_L$ contains a fully retained
coordinate cycle of length $L$.  Since $f(N)=o(L)$, the component containing
this cycle has order at least $L>f(N)$ for all sufficiently large $L$.  Hence
that component survives in $Y_L$, and $Y_L$ contains an odd cycle with
probability tending to one.  Therefore
\[
        \Prob(Y_L\models\beta)\longrightarrow0
\]
along odd $L$.

Thus the probabilities of the same $\mathrm{MSO}_1$ sentence have different
subsequential limits, so $(Y_L)$ has neither an $\mathrm{MSO}_1$ nor an
$\mathrm{MSO}_2$ convergence law.
\end{proof}

\begin{remark}\label{rem:mso-torus}
For example, Proposition~\ref{prop:mso-torus-obstruction} applies when
$p_N=N^{-\alpha}$ with $\alpha>1/d$.  The obstruction is global rather than
local: no single first-order sentence uniformly detects the parity of these
unbounded coordinate cycles, whereas monadic second-order logic can express
bipartiteness.  This does not rule out monadic
limit laws for special subsequences, such as even side
lengths, or under additional global hypotheses.  It only shows that the torus
theorems proved above do not have a direct monadic analogue, even in
$\mathrm{MSO}_1$, when $L$ ranges over all integers.
\end{remark}

\begin{remark}
The arguments follow a common pattern.  First, logical questions are reduced
to finite local or component data: Hanf locality does this on tori, while the
pruned Erd\H{o}s--R\'enyi graph is handled by finite component types.  The
relevant witnesses are then counted.  Away from a threshold, each type is
absent or saturated; at a threshold, the underlying defect clusters have
Poisson limits.

For a fixed first-order torus sentence, only finitely many quantities
$N p_N^b$ (or $N q_N^b$) arise from the Hanf radius.  The polynomial
classification in Theorem~\ref{thm:polynomial} packages this finite collection
of thresholds into the usual exponent notation.  The monadic obstruction is
global and therefore does not rule out monadic convergence on special
subsequences, such as even side lengths, or in modified models that suppress
the relevant parity phenomenon.
\end{remark}

\end{document}